\documentclass[12pt,a4paper]{amsart}
\usepackage{mathtools}
\numberwithin{equation}{section}
\usepackage{amssymb}
\usepackage{amsmath}
\newtheorem{lemma}{Lemma}
\newtheorem*{thm}{Theorem}
\theoremstyle{remark}
\newtheorem{rem}{Remark}
\newtheorem*{ack}{Acknowledgment}
\def\R{\mathbb{R}}
\def\N{\mathbb{N}}
\def\D{\mathbb{D}}
\def\C{\mathbb{C}}
\def\dist{\operatorname{dist}}
\def\re{\operatorname{Re}}
\begin{document}
\title[No fixed points and  no Baker domains]{An entire function with no fixed points and
no invariant Baker domains}
\author{Walter Bergweiler}
\thanks{Supported by the ESF Research Networking Programme HCAA}
\address{Mathematisches Seminar,
Christian--Albrechts--Universit\"at zu Kiel,
Lude\-wig--Meyn--Str.~4,
D--24098 Kiel,
Germany}
\email{bergweiler@math.uni-kiel.de}
\subjclass[2010]{Primary 37F10;  Secondary 30D05, 49M15, 65H05}
\begin{abstract}
We show that there exists an entire function which has neither
fixed points nor invariant Baker domains. The question whether
such a function exists was raised by Buff.
\end{abstract}
\maketitle
\section{Introduction and result} \label{intro}
Let $f$ be a meromorphic function and denote by $f^n$ the $n$-th iterate
of~$f$. An invariant component $U$ of the Fatou set of $f$ such that
$f^n|_U\to\infty$ as $n\to\infty$ is called an \emph{invariant Baker
domain}; cf.~\cite[\S 4.7]{Ber93} or~\cite{Rip}. It was suggested by 
Douady that invariant Baker domains of the Newton function $f(z)=z-g(z)/g'(z)$
of an entire function $g$ are related to paths where $g$  
tends to the asymptotic value~$0$.
In response to Douady's question it was shown 
in~\cite{BufRuc} that under mild additional hypotheses
the existence of an invariant Baker domain does indeed imply
that $0$ is an asymptotic value of~$g$.
However, this is not always the case~\cite{BerDraLan}.

If $g$ has no zeros at all, then the Newton function $f$ has no fixed points.
Moreover, $0$ is an asymptotic value of $g$ by Iversen's theorem~\cite[p.~289]{Nev}.
This led Buff to ask whether there exists an entire function having no
fixed points and no invariant Baker domains.  We show that such a function exists.
\begin{thm}
There exists an entire function with no fixed points and no invariant Baker domains.
\end{thm}
A meromorphic function with this property was constructed in~\cite{Ber10}.
The present  construction is based on similar ideas.
As in~\cite{Ber10}, a function $f$ satisfying the conclusion of the
theorem can be given explicitly.

Let $(r_k)$ be a sequence of real numbers
tending to $\infty$ and let $(n_k)$ be a sequence of positive 
integers satisfying $n_k\geq k$ for all $k\in\N$. Then 
\[
h(z) = \prod_{k=1}^\infty \left(1+\left(\frac{z}{r_k}\right)^{n_k} \right)
\]
defines an entire function~$h$.
Indeed, if $|z|\leq R$ and $k$ is so large that $r_k\geq 2R$, then
$|z/r_k|^{n_k}\leq 2^{-k}$, implying that the infinite product converges
locally uniformly.

For $k\geq 2$ we put $m_k= \sum_{j=1}^{k-1} n_j$.  We shall show that if 
\begin{equation}
\label{1b}
r_k\geq 2 r_{k-1}\geq 4 
\quad  \text{and}\quad
n_k\geq 20 r_k^2\exp\left(4r_k^{m_{k}}\right)
\end{equation}
for $k\geq 2$,
then $f(z)=z+e^{h(z)}$ has the required property.
\begin{ack}
I thank Alexandre Eremenko for useful comments.
\end{ack}
\section{Preliminaries}
\subsection{The hyberbolic metric}
We need some standard results about the hyperbolic metric which 
can be found in, e.g.,~\cite[Section I.4]{CarGam}.

We denote the open disk of radius $r$ around a point $c\in\C$ by
$D(c,r)$ and put $\D=D(0,1)$.
The density of the hyperbolic metric in a hyperbolic
domain $U$ is denoted by~$\lambda_U$, normalized such that
$\lambda_{\D}(z)=2/(1-|z|^2)$.
The hyperbolic metric is denoted by $\rho_U$. 
For $a,b\in U$ we thus have
\[
\rho_U(a,b)=\inf_\gamma\int_\gamma\lambda_U(z)|dz|,
\]
where the infimum is taken over all curves $\gamma$ that connect $a$ and~$b$.
Then~\cite[p.~11]{CarGam}
\begin{equation}\label{5a}
\rho_\D(0,z)=\log\frac{1+|z|}{1-|z|} \quad\text{for}\ z\in\D.
\end{equation}
It follows from Schwarz's lemma and the Koebe one quarter theorem that if
$U$ is simply connected, then~\cite[Theorem I.4.3]{CarGam}
\begin{equation}\label{5b}
\frac{1}{2\dist(z,\partial U)}
\leq \lambda_U(z) \leq
\frac{2}{\dist(z,\partial U)}
\end{equation}
for all $z\in U$. Here $\dist(z,\partial U)=\inf_{\zeta\in\partial U}|\zeta-z|$.

The following lemma is a simple consequence of~\eqref{5b}.
\begin{lemma}\label{lemma1}
Let $U$ be a simply connected hyperbolic domain, $a,b\in U$ and $c\in\C\backslash U$. 
Then
\[
\rho_U(a,b)\geq \frac{1}{2}\left|\log\left|\frac{b-c}{a-c}\right|\right|.
\]
\end{lemma}
\begin{proof}
Without loss of generality we may assume that $c=0$.
Let $\gamma$ be a curve from $a$ to $b$ and let $L$ be a branch of
the logarithm defined in~$U$. Then~\eqref{5b} yields
\[
\begin{aligned}
\int_\gamma\lambda_U(z)|dz|
&\geq
\frac{1}{2} \int_\gamma \frac{|dz|}{\dist(z,\partial U)}
\geq
\frac{1}{2} \int_\gamma \frac{|dz|}{|z|}
\geq
\frac{1}{2} \left|\int_\gamma \frac{dz}{z}\right|
\\ &
=\frac12 \left|L(b)-L(a)\right|
\geq \frac12 \left|\re(L(b)-L(a))\right|
=
\frac{1}{2}\left|\log \left|\frac{b}{a}\right|\right| ,
\end{aligned}
\]
from which the conclusion follows.
\end{proof}
The next lemma follows easily from~\eqref{5a} and the triangle inequality.
\begin{lemma}\label{lemma2}
If $a,b\in D(c,r/2)$, then $\rho_{D(c,r)}(a,b)\leq 2\log 3$.
\end{lemma}
Finally we have the following form of Schwarz's lemma~\cite[Theorem I.4.3]{CarGam}.
\begin{lemma}\label{lemma3}
Let $U,V$ be hyperbolic domains,
$f\colon U\to V$ holomorphic and $a,b\in U$. Then
$\rho_V(f(a),f(b))\leq \rho_U(a,b)$.
\end{lemma}
Applying this lemma to $f(z)=z$ yields
\begin{equation}\label{2z}
\rho_V(a,b)\leq \rho_U(a,b)\quad  \text{if } U\subset V.
\end{equation}

\subsection{Some growth estimates}
We have to estimate the growth of $h$ on certain circles from above
and below.
For $k\geq 2$ and $|z|=r_k$ we have
\[
\begin{aligned}
\log |h(z)|
&\leq 
\sum_{j=1}^{k-1} \log\left(1+\left(\frac{r_k}{r_j}\right)^{n_j}\right)
+\log 2 
+\sum_{j=k+1}^\infty \log\left(1+\left(\frac{r_k}{r_j}\right)^{n_j}\right)\\
&\leq
\sum_{j=1}^{k-1} \log\left(1+\frac12 {r_k}^{n_j}\right)
+\log 2
+\sum_{j=k+1}^\infty \left(\frac{r_k}{r_j}\right)^{n_j}\\
&\leq 
\sum_{j=1}^{k-1} \log\left({r_k}^{n_j}\right)
+\log 2
+\sum_{j=k+1}^\infty 2^{-n_j}
\leq 
m_k\log r_k +2\log 2.
\end{aligned}
\]
Hence
\begin{equation}
\label{2a}
|h(z)|\leq 4 r_k^{m_{k}}
\quad\text{for}\ |z|=r_k
\end{equation}
and $k\geq 2$.

We put $s_k
=
(1+1/n_k)r_k$. For $t\in[0,2\pi]$ and $z=s_ke^{it}$ we have
\[
\begin{aligned}
h(z)
&=
\prod_{j=1}^{k-1} \left(1+\left(\frac{z}{r_j}\right)^{n_j}\right)
\cdot \left(1+\left(1+\frac{1}{n_k}\right)^{n_k} e^{in_kt}\right)
\cdot\prod_{j=k+1}^\infty \left(1+\left(\frac{z}{r_j}\right)^{n_j}\right)\\
&\sim 
\prod_{j=1}^{k-1} \left(\frac{z}{r_j}\right)^{n_j}
\left(1+e\cdot e^{in_kt}\right)
= 
\prod_{j=1}^{k-1} \left(\frac{s_k}{r_j}\right)^{n_j}
 e^{im_kt} \left(1+e\cdot e^{in_kt}\right)
\end{aligned}
\]
as $k\to\infty$. Putting
\begin{equation}
\label{2a1}
T_k
= \prod_{j=1}^{k-1} \left(\frac{s_k}{r_j}\right)^{n_j}
\end{equation}
we thus have
\begin{equation}
\label{2b}
h(s_ke^{it})\sim T_k  e^{im_kt} \left(1+e\cdot e^{in_kt}\right)
\end{equation}
as $k\to\infty$.

It is not difficult to see that
 for each $\varphi\in\R$ there exists $\theta=\theta(\varphi)\in[0,1]$
such that
$e^{2\pi i\varphi} \left(1+e\cdot e^{2\pi i \theta}\right)$ is positive.
For $\nu\in\{0,1,\dots,n_k-1\}$ we put 
\[
\theta_\nu =\theta(\nu m_k/n_k)
\quad\text{and} \quad
p_\nu
=
e^{2\pi i\nu m_k/n_k} \left(1+e\cdot e^{2\pi i \theta_\nu}\right).
\]
Then $p_\nu$ is positive and thus $p_\nu=|p_\nu| \geq e-1$.
Since $m_k/n_k\to 0$ by~\eqref{1b},
we deduce from~\eqref{2b}  that
\[
h(s_ke^{2\pi i(\nu+\theta_\nu)/n_k})
\sim 
T_k e^{2\pi i(\nu+\theta_\nu)m_k/n_k} \left(1+e\cdot e^{2\pi i (\nu+\theta_\nu)}\right)
= 
p_\nu T_k e^{2\pi i\theta_\nu m_k/n_k} 
\sim
p_\nu T_k
\]
Thus
\begin{equation}
\label{2c}
\re h(s_ke^{2\pi i(\nu+\theta_\nu)/n_k})
\geq T_k
\end{equation}
for large $k$ and all 
$\nu\in\{0,1,\dots,n_k-1\}$.
\section{Proof of  the theorem}
Let $f$ be as defined in the introduction and
suppose that $f$ has an invariant Baker domain~$U$. 
By a result of Baker~\cite{Baker1975}, $U$ is simply connected.
Take $z_0\in U$, connect $z_0$ and $f(z_0)$ by a curve $\gamma_0$ in $U$
and put $\gamma= \bigcup_{j=0}^\infty f^j(\gamma_0)$. Then $\gamma$
is a curve in $U$ connecting $z_0$ to~$\infty$.
As $\gamma_0$ is compact, there exists $K> 0$ such that 
$\rho(f(z),z)\leq K$ for all $z\in\gamma_0$.
Since every $z\in\gamma$ has the form $z=f^j(\zeta)$ for some
$\zeta\in\gamma_0$ and some $j\geq 0$,
Lemma~\ref{lemma3} yields
\begin{equation}
\label{3z}
\rho(f(z),z)\leq K\quad  \text{for} \  z\in\gamma.
\end{equation}
For large $k$ the curve $\gamma$ intersects the
circle $\{z\colon |z|=r_k\}$. Let
$z_k$ be a point of intersection.

We shall show first  that if $k$ is large enough,
then the disk $D(z_k,20r_k/n_k)$ 
is not contained in~$U$; that is,
\begin{equation}
\label{3y}
D(z_k,20r_k/n_k)\cap \partial U\neq \emptyset.
\end{equation}
In order to do so we assume that $D(z_k,20r_k/n_k)\subset U$.
We write $z_k=r_ke^{2\pi i t_k}$ with $t_k\in[0,1)$ and put
$\nu= [n_kt_k]$, where $[x]$ denotes the largest integer not greater
than~$x$. 
Thus $n_kt_k=\nu+\delta$ where
$\nu\in\{0,1,\dots,n_k-1\}$ and
$\delta\in[0,1)$.
Let 
\[
a_k=r_ke ^{(2\nu+1)\pi i/n_k}
\quad \text{and} \quad
b_k=s_ke^{2\pi i (\nu+\theta_\nu)/n_k}.
\]
Then 
\[
|a_k-z_k|=r_k\left|e^{(2\nu+1)\pi i/n_k-2\pi i t_k}-1\right|
=r_k\left|e^{(1-2\delta)\pi i/n_k}-1\right|
\sim \frac{|1-2\delta|\pi r_k}{n_k}
\]
and 
\[
\begin{aligned}
|b_k-z_k|
&\leq |b_k- r_ke^{2\pi i (\nu+\theta_\nu)/n_k}|
+|r_ke^{2\pi i (\nu+\theta_\nu)/n_k}-z_k|
\\ &
= 
s_k-r_k
+r_k\left|e^{2\pi i(\nu+\theta_\nu)/n_k-2\pi i t_k}-1\right|
\\ &
= \frac{r_k}{n_k}
+r_k\left|e^{2\pi i(\theta_\nu-\delta)/n_k}-1\right|
\\ &
\sim \frac{(1+2\pi|\theta_\nu-\delta|)r_k}{n_k},
\end{aligned}
\]
which implies that
\[
a_k\in D(z_k,10 r_k/n_k)
\quad\text{and}\quad 
b_k\in D(z_k,10 r_k/n_k)
\]
for large~$k$.
Lemma~\ref{lemma2} and~\eqref{2z} now yield
\begin{equation}
\label{3b}
\rho_U(a_k,b_k)\leq \rho_{D(z_k,20r_k/n_k)}(a_k,b_k)\leq 2\log 3.
\end{equation}
Since $h(a_k)=0$ by the definition of $h$ and $\re h(b_k)\geq T_k\geq s_k/r_1$
by~\eqref{2a1} and~\eqref{2c}, we have
\begin{equation}
\label{3c}
|f(a_k)|=|a_k+1|\leq r_k+1
\quad\text{and}\quad 
|f(b_k)|\geq 
e^{s_k/r_1}-s_k\geq s_k^2\geq r_k^2
\end{equation}
for large $k$.
Fix a point $c\in\partial U$.
Lemma~\ref{lemma1} and~\eqref{3c} imply that 
\begin{equation} \label{3d}
\rho_U(f(a_k),f(b_k))\geq 
\frac12 \log \left|\frac{f(b_k)-c}{f(a_k)-c}\right|
\geq 
\frac12 \log \frac{r_k^2 -|c|}{r_k+1+|c|}
\end{equation}
for large~$k$.
Now a  contradiction is obtained from Lemma~\ref{lemma3},
\eqref{3b} and~\eqref{3d}, provided $k$ is sufficiently large.
This contradiction shows that~\eqref{3y} holds for large~$k$.

Thus, for large $k$, there exists $c_k\in
D(z_k,20r_k/n_k)\cap \partial U$.
Lemma~\ref{lemma1}
now yields 
\[
\rho_U(f(z_k),z_k)
\geq 
\frac12 \log \left|\frac{f(z_k)-c_k}{z_k-c_k}\right|
= 
\frac12 \log \left|\frac{e^{h(z_k)}}{z_k-c_k}+1\right|.
\]
Since 
\[
\left|\frac{e^{h(z_k)}}{z_k-c_k}\right|
\geq
\frac{e^{- |h(z_k)|}}{|z_k-c_k|}
\geq 
\frac{n_k\exp\left(-4r_k^{m_k}\right)}{20r_k}
\geq r_k
\]
for $k\geq 2$ by~\eqref{1b} and~\eqref{2a}, 
we obtain
\[
\rho_U(f(z_k),z_k)
\geq \frac12 \log(r_k-1)
\]
for large $k$, contradicting~\eqref{3z}.\qed

\begin{rem}
Buff and R\"uckert~\cite{BufRuc} considered \emph{virtual immediate basins}
instead of invariant Baker domains. However, for functions for which
 all Baker domains are simply connected the two concepts coincide; cf.\ the discussion
in~\cite[p.~431]{Ber10} or~\cite[p.~4]{BufRuc}.
By the result of Baker~\cite{Baker1975} already quoted, this holds in particular
for entire functions. 
By a recent result of 
Bara\'nski, Fagella, Jarque and Karpi\'nska~\cite{BFJK}, it also holds 
for Newton maps of entire functions.
\end{rem}
\begin{rem}
The function $f$ constructed in the proof is the Newton function for
\[
g(z)=\exp\!\left(-\int_0^z e^{-h(t)}dt\right).
\]
Since $g$ has no zeros, $g$ has a direct singularity over~$0$; 
see~\cite[p.~289]{Nev} for this result,
as well as \cite{BerEre93,Six} for 
the terminology used here and below.
As $f$ has no invariant Baker domains, 
$g$ has no logarithmic singularity over~$0$ by
one of the results obtained by Buff and R\"uckert in the paper 
already mentioned in the introduction~\cite[Theorem~4.1]{BufRuc}.
Thus $g$ has a direct non-logarithmic singularity over~$0$.
This implies~\cite[Theorem~5]{BerEre08}
that $g$ has uncountably many direct non-logarithmic 
singularities over~$0$.
As $g$ has no critical points, a result of Sixsmith~\cite[Theorem~ 1.2]{Six}
yields that every neighborhood of any of these singularities contains a 
neighborhood of an indirect or logarithmic singularity of $g$
whose projection is different from~$0$.
Overall we see that the set of singularities of the
inverse of $g$ has a quite complicated structure.
\end{rem}

\end{document}